\documentclass[amsfonts]{amsart}
\usepackage{amsmath,dsfont,amsfonts}
\usepackage{amssymb}
\usepackage{amsthm}
\usepackage{verbatim}
\usepackage{graphicx}
 \numberwithin{equation}{section}

\usepackage{enumerate}
\theoremstyle{plain}
\newtheorem{theo}{Theorem}[section]
\newtheorem{lem}[theo]{Lemma}
\newtheorem{prop}[theo]{Proposition}
\newtheorem{claim}[theo]{Claim}
\newtheorem{definition}[theo]{Definition}
\theoremstyle{remark}
\newtheorem{rem}[theo]{Remark}
\newcommand{\RR}{{\mathbb{R}}}

\newcommand{\der}{\partial}
\newcommand{\ep}{\varepsilon}
\newcommand{\ov}{\overline}

\newcommand{\om} {\Omega}
\newcommand{\dive}{\text{\normalfont div}}

\def\qed{\hfill$\square$\vspace{0.5cm}}    

\begin{document}

\title[Stability for qPAT]{Stability for Quantitative Photoacoustic Tomography with well chosen illuminations}
\author[G. Alessandrini, M. Di Cristo, E. Francini, S. Vessella]{Giovanni Alessandrini}\address{Universit\`a di Trieste}\email{alessang@units.it}
\author[]{Michele Di Cristo}\address{Politecnico di Milano}\email{michele.dicristo@polimi.it}\author[]{Elisa Francini}\address{Universit\`a di Firenze}\email{elisa.francini@unifi.it}
\author[]{
Sergio Vessella}\address{Universit\`a di Firenze}\email{sergio.vessella@unifi.it}

\begin{abstract}
We treat the stability issue for the three dimensional inverse imaging modality called Quantitative Photoacoustic Tomography. We provide universal choices of the illuminations which enable to recover, in a H\"older stable fashion, the diffusion and absorption coefficients from the interior pressure data. With such choices of illuminations we do not need the nondegeneracy conditions commonly used in previous studies, which are difficult to be verified a-priori. \end{abstract}
\maketitle

\section{Introduction}
In recent years there has been an increasing interest towards  imaging methods which combine different physical modalities of interrogation which are known as hybrid, or coupled physics, inverse problems.


In this paper we concentrate on Photoacoustic Tomography (PAT) that couples the (high contrast) optical tomography with the (high resolution) ultrasound waves.
The first step of this imaging technique consists of reading the boundary response to an acoustic signal, in order to reconstruct
the absorbed energy distribution inside the biological tissue under inspection. We assume this first part as already performed and we focus on the second step of the procedure, which consists of reconstructing the
absorption and diffusion coefficients from measurements of the absorbed energy distribution obtained in the previous step. We refer to \cite{RGZ} for an extended bibliographical review on this problem known as Quantitative Photoacoustic Tomography (qPAT).

Let us denote by $\Omega\subset\mathbb R^n$  the body enclosing the biological tissue under inspection. Note that the physically significant space dimension is $n=3$, but, since we shall need also to refer to model cases when $n=2$, we shall leave $n\geq 2$ undetermined. We denote by  $u(x)$ the photon density at the point $x\in\om$.
Then $u(x)$ solves the following boundary value problem
\begin{equation}
\label{m1}
\left\{\begin{array}{ll}
-\dive( D\nabla u)+\sigma u=0 & \textrm{in }\Omega,\\[2mm]
u=g & \textrm{on }\partial \Omega,
\end{array}\right.\end{equation}
where
$D=D(x)>0$,  $\sigma=\sigma(x)$ are the diffusion and absorption coefficients, respectively,
and $g$ is the illumination source prescribed on the boundary.

The goal of qPAT is to recover information on the coefficients $D$, $\sigma$ from the knowledge
of the absorbed energy
\[H(x)=\sigma(x)u(x),\qquad x\in\ov\Omega,\]
possibly repeating the experiment with different profiles of the illumination $g$.

Note, incidentally, that a more accurate model would require the introduction of the additional multiplicative unknown parameter $\Gamma(x)$ called the Gr\"uneisen coefficient.
Here we adopt, for simplicity, the commonly used convention of assuming $\Gamma\equiv 1$. See \cite{Ba-Re} for a discussion on this issue.

This problem has been considered in \cite{Ba-Uh}, where a uniqueness result with two measurements is proven.
The authors assume $D\in C^{k+2}$, $\sigma\in C^{k+1}$ with $k\geq1$.
They also present  a Lipschitz stability theorem.
This result has been later improved in \cite{Ba-Re} also providing a numerical reconstruction
procedure.
More recently, \cite{RGZ} have considered the same problem when the prescribed illumination is modeled by a  Robin boundary condition, rather than a Dirichlet one. A reconstruction method
for the full nonlinear inversion is also treated.

All the above quoted results rely on a nondegeneracy condition which can be illustrated as follows. If $u_1$, $u_2$ are two solutions to \eqref{m1} corresponding to two different illuminations $g_1$, $g_2>0$, then it is a quite well-known fact that the quotient
\[u=\frac{u_2}{u_1}\]
satisfies an elliptic equation in pure divergence form
\begin{equation*}
\left\{\begin{array}{ll}
\dive( a\nabla u)=0 & \textrm{in }\Omega,\\[2mm]
u=g & \textrm{on }\partial \Omega,
\end{array}\right.\end{equation*}
where $g$ is the ratio $\frac{g_2}{g_1}$ (see Proposition \ref{proph-1} below). It is also easy to see that the solution of the qPAT problem boils down to solve the inverse problem of finding $a$ given $u$ and the boundary values $a_{|_{\der\om}}$. This is a relatively easy task if $g$ is chosen in such a way that the following nondegeneracy condition holds:
\begin{equation}\label{X}
|\nabla u|>0,\textrm{ everywhere in  }\om.
\end{equation}
When $n=2$ there exists a well-established criterion which enables to choose the Dirichlet data $g$ (independently of $a$!) so that \eqref{X} holds (\cite{A}, \cite{AM}).
Such a criterion is unimodality. That is, roughly speaking that the graph  of $g$ has one single peak of maximum points, one of minimum and is monotone in between.

In dimension $n\geq 3$ complex valued solutions satisfying \eqref{X} can be constructed by the method of Complex Geometrical Optics \cite{Ba-Uh}, but their boundary data do depend on the interior values of the (unknown) coefficient $a$ and thus they cannot be a-priori chosen.

Real valued solutions which locally satisfy $|\nabla u|>0$ can also be constructed (see \cite{Greene-Wu}, \cite[Theorem 4.7]{Bal-Uhl}), but still they depend on the unknown coefficient $a$.

Indeed, there are reasons to believe that, when $n\geq 3$, it does not exists any Dirichlet data $g$ such that \eqref{X} is satisfied for every $a$. See, for related discussions, \cite{AN}, \cite{Ca}.

The principal aim of the present paper is to treat stability even when the above stated nondegeneracy condition may be violated.

We shall show that stability can be obtained with essentially arbitrary illuminations $g_1$, $g_2$ imposing only one constraint on their ratio $g=\frac{g_2}{g_1}$. Namely that $g$ satisfies a condition of unimodality adapted to the $(n-1)$-dimensional boundary $\der\om$. This condition shall be made more precise in Definition \ref{qu}.

Our strategy shall be as follows. In dimension $n\geq 3$, for a fixed $g$, we cannot assure the nonvanishing of $|\nabla u|$ throughout $\om$, but, under reasonable assumptions, it is possible to keep under control the vanishing rate in the interior. This will be the content of Lemma \ref{quc}.

On the other hand, assuming unimodality we can make sure that $|\nabla u|>0$ on a small neighborhood of $\der\om$ (Lemma \ref{Hopf}).

Next we adapt from \cite{A} a weighted stability estimate on the coefficients $a$ in terms of $u$ and of $a_{|_{\der\om}}$.

Using the previously mentioned estimates on the vanishing rate of $|\nabla u|$ and a suitable interpolation inequality \cite{S}, we arrive at an (unweighted) stability estimate of H\"older type for $a$ (Theorem \ref{teo2}). The deduction of stability bounds for $D$ and $\sigma$ follows the track of well-known arguments, see for instance \cite{RGZ}.

Let us emphasize that most of the present effort is devoted to two main goals:
\begin{enumerate}
\item To avoid the nondegeneracy condition \eqref{X}.
\item To make precise (but feasible) a-priori assumptions which guarantee a quantitative, concrete, evaluation in our stability estimates.
\end{enumerate}
It is our belief that the present approach can be useful also to other, more complex, hybrid inverse problems, where analogous issues of nondegeneracy arise.

The paper is organized as follows. In the next Section \ref{Sec2} we provide the main assumptions
and we state our main result (Theorem \ref{theoh-2}). The proof of it is based on some auxiliary propositions,
given in the subsequent Section \ref{Sec3}, and  is presented in Section \ref{Sec4}.

\section{Assumptions and Main Result}
\label{Sec2}
We assume $\om$ to be a $C^2$ - smooth, bounded domain in $\RR^n$ diffeomorphic to the unit ball $B_1(0)$.
More precisely, from a quantitative point of view, we assume that there exists a diffeomorphism
$F$ of class $C^2$ such that, for given constants $Q_0$, $Q_1>0$,
\begin{subequations}\label{diff}
\begin{equation}
F:B_1(0)\leftrightarrow\om,\end{equation}
\begin{equation}
\|F\|_{C^2(B_1(0))}\leq Q_0,\end{equation}
\begin{equation} \left|F(x)-F(y)\right|\geq \frac{1}{Q_1} |x-y| ,\textrm{ for every } x,y\in B_1(0).\end{equation}
\end{subequations}
The constants $Q_0$, $Q_1$, shall be part of the a-priori information that shall be used in our quantitative stability estimates.

We are interested in recovering the unknown parameters $D$ and $\sigma$, by  performing two measurements, that is prescribing two data $g_1$ and $g_2$ on $\der\om$ and measuring the corresponding internal pressure fields. In particular we  establish a continuous dependence of the unknown parameters  from the measured data.

Given constants $\lambda_0$, $\lambda_1$, $E_0$, $E_1$, $\mu_0$, $\mu_1$, $\mu_2>0$ (which also shall be part of the a-priori information), we consider unknown coefficients $D$ and $\sigma$ such that
\begin{equation}
	\label{h1-1}
	D\in W^{1,\infty}(\om),\quad \sigma\in W^{1,\infty}(\om),
\end{equation}
and
\begin{equation}
	\label{h2-1}
	\lambda_0^{-1}\leq D\leq \lambda_0, \textrm{ for every }x\in \om,
	\end{equation}

\begin{equation}
	\label{h4-1}
	\lambda_1^{-1}\leq \sigma\leq \lambda_1, \textrm{ for every }x\in \om,
	\end{equation}
	
	\begin{equation}
\label{h3-1}
\|D\|_{W^{1,\infty}(\om)}\leq E_0,\quad		\|\sigma\|_{W^{1,\infty}(\om)}\leq E_1.
	\end{equation}
The boundary data we choose are functions $g_1$ and $g_2$ such that
\begin{equation}
	\label{h4.5-1}
g_i\in C^2(\der\om), \quad\quad \|g_i\|_{C^2(\der\om)}\leq \mu_0\,\mbox{ for }i=1,2,\end{equation}
\begin{equation}
	\label{h5-1}
	\mu_1^{-1}\leq g_i(x)\leq \mu_1,\textrm{ for every } x\in\om,\quad i=1,2.
\end{equation}
Moreover, denoting by
$$g=\frac{g_2}{g_1}\quad\textrm{and}\quad\ov g=\frac{1}{|\partial \om|}\int_{\partial\om}g,$$
we assume that
\begin{equation}
\label{h5.5-1}
\|\,g-\ov{g}\,\|_{L^2(\der\om)}\geq \mu_2^{-1}.
\end{equation}
\begin{rem}
Let us emphasize that \eqref{h5.5-1} represents constructively the assumption that the illuminations $g_1$, $g_2$ are linearly independent.
\end{rem}
\begin{rem}
Note that, if $g_1$, $g_2$ satisfy assumptions \eqref{h5-1} and \eqref{h5.5-1}, then the so-called frequency function associated to $g$
\begin{equation}
	\label{freq}
	F[g]:=\frac{\|g-\ov{g}\|_{H^{1/2}(\der\om)}}{\|g-\ov{g}\|_{L^{2}(\der\om)}}
\end{equation}
is bounded by a constant depending only on $\mu_0$, $\mu_2$ and $Q_1$.
\end{rem}

We shall see that it is convenient to assume that the ratio $g=\frac{g_2}{g_1}$ of the illuminations has a specific behaviour which is expressed in the following definition. This is a form of monotonicity assumption, which needs however to be specified in quantitative fashion.
\begin{definition}\label{qu}
Given $m$, $M$, $0<m<M$ and a continuous, strictly increasing function $\omega:\RR^+\to\RR^+$, such that  $\omega(0)=0$,
we say that a function $g\in C^1(\der\om,\RR)$ is \emph{quantitatively unimodal} if
\begin{equation}\label{maxmin}
m\leq g(x)\leq M,\mbox{ for every }x\in \der\om,
\end{equation}
the subsets of $\der\om$
\begin{equation}\label{ptimaxmin}
\Gamma_m=\{x\in\der\om\,:\,g(x)=m\}\mbox{ and }\Gamma_M=\{x\in\der\om\,:\,g(x)=M\}
\end{equation}
are connected and non-empty, possibly reduced to  single points, and,
for every $x\in\der\om\setminus(\Gamma_m\cup\Gamma_M)$ such that $dist(x,\Gamma_m\cup\Gamma_M)\geq \delta$, we have
\begin{equation}\label{quantunimod}
|\nabla_T g(x)|\geq \omega(\delta),
\end{equation}
where $\nabla_T$ denotes the tangential gradient.
\end{definition}
Also $m$, $M$, and $\omega$ shall be part of the a-priori information.

In our stability result, we compare two different sets of diffusion and absorption coefficients, that we denote by $D^{(1)},\,\sigma^{(1)} $ and $D^{(2)},\, \sigma^{(2)}$.

Let $u_i^{(j)}$, for $i,j=1,2$, be solution to
\begin{equation}\label{h6-1}
	\left\{
  \begin{array}{ll}
    -\dive\left(D^{(j)}\nabla u_i^{(j)}\right)+\sigma^{(j)}u_i^{(j)} =0& \mbox{in }\om,\\
    u_i^{(j)}=g_i& \mbox{on }\der\om.
  \end{array}
\right.
\end{equation}
We emphasize that, in the notation $u_i^{(j)}$, the superscript is associated to the unknown parameters $D^{(j)}$, $\sigma^{(j)}$, whereas the subscript is associated to the illumination $g_i$.

The available measurements which represent the internal pressure fields generated by the absorptions of photons energy are given by
\begin{equation}
	\label{h8-1}
	H_i^{(j)}=\sigma^{(j)}u_i^{(j)}.
\end{equation}


We prove the following (uniqueness and) stability result:
\begin{theo}\label{theoh-2}
Let all the assumptions stated above be satisfied.
If
\begin{equation}
	\label{h1-2}
	\left\|H_i^{(1)}-H_i^{(2)}\right\|_{L^2(\om)}\leq \ep,\mbox{ for }i=1,2,
\end{equation}
and
\begin{equation}\label{Dbordo}
\left\|D^{(1)}-D^{(2)}\right\|_{L^\infty(\der\om)}\leq \ep^\prime,
\end{equation}
then we have
\begin{equation}
	\label{h2-2}
	\left\|D^{(1)}-D^{(2)}\right\|_{L^2(\om)}+\left\|\sigma^{(1)}-\sigma^{(2)}\right\|_{L^2(\om)}
\leq C\left(\ep+ \ep^\prime\right)^{\theta},
\end{equation}
where $C$ and $\theta\in(0,1)$ only depend on the a-priori information $Q_0$, $Q_1$, $\lambda_0$, $\lambda_1$, $E_0$, $E_1$, $\mu_0$, $\mu_1$, $\mu_2$, $m$, $M$ and $\omega$.
\end{theo}

%

\section{Auxiliary results}
\label{Sec3}
The proof of Theorem \ref{theoh-2} is based on a result concerning stable reconstruction of the main coefficient of a second order elliptic equation from the knowledge of internal values of one of its non constant solutions.

\begin{theo}\label{teo1} Let $\om$ be diffeomorphic to the unit ball.
Let $a$ and $b\in W^{1,\infty}(\om)$ such that
\begin{equation}
	\label{3-1}
	C_0^{-1}\leq a(x),\,b(x)\leq C_0,\textrm{ for every }x\in \om,
\end{equation}
and
\begin{equation}
	\label{3.5-1}
	\left|\nabla a(x)\right|, \left|\nabla b(x)\right|\leq C_1 ,\textrm{ for almost every }x\in \om.
\end{equation}
Let $g$ and $k\in C^2(\der\om)$, with
\begin{equation}
\label{4-1}
\|g\|_{C^2(\der\om)},\|k\|_{C^2(\der\om)}\leq C_2.\end{equation}
Assume $g$ satisfies assumption \eqref{h5.5-1} and
\begin{equation}
	\label{5-1}
	g(x)\geq C_3^{-1}\mbox{ for every } x\in\om.
\end{equation}

Let $u$ and $v$ in $W^{1,2}(\om)$ be the unique solutions to boundary value problems
\begin{equation}\label{Pu}
	\left\{
  \begin{array}{rl}
    \dive\left(a(x)\nabla u(x)\right) =0& \mbox{in }\om,\\
    u=g& \mbox{on }\der\om,
  \end{array}
\right.
\end{equation}
and
\begin{equation}\label{Pv}
	\left\{
  \begin{array}{rl}
    \dive\left(b(x)\nabla v(x)\right) =0& \mbox{in }\om,\\
    v=k& \mbox{on }\der\om.
  \end{array}
\right.
\end{equation}
Given $\om^\prime\subset\subset\om$ with $\mbox{dist}\left(\der\om,\om^\prime\right)\geq d_0$ and $\theta\in(0,1/2)$ there are positive constants
$\tilde{C}$ and $\beta$, depending only on $C_0$, $C_1$, $C_2$, $\mu_2$, $C_3$, $Q_0$, $Q_1$, $d_0$ and $\theta$,  such that
\begin{equation}
	\label{tesiteo1}
\|a-b\|_{L^\infty(\om^\prime)}\leq \tilde{C}\left(\|u-v\|_{L^2(\om)}^\theta+\|a-b\|_{L^\infty(\der\om)}\right)^\beta.
\end{equation}
\end{theo}

In order to prove Theorem \ref{teo1}, we need the following lemmas.
For $d>0$ we will denote by $\om_d=\{x\in\om\,:\, dist(x,\der\om)>d\}$.
\begin{lem}\label{lemma2.2}
With the same assumptions of Theorem \ref{teo1}, for every $\rho>0$ and for every $x\in\om_{4\rho}$
\[\int_{B_\rho(x)}|\nabla u|^2\geq C_\rho \int_{\om}|\nabla u|^2,\]
where $C_\rho$ depends on $C_0$, $C_1$, $Q_0$, $Q_1$, $F[g]$  and $\rho$ only.
\end{lem}
\begin{proof}
This Lemma corresponds to Theorem 4.2 in \cite{AMR2003} for solutions of Dirichlet boundary problem instead of Neumann boundary problem. The proof follows the same path by taking $u_0=u-\ov{u}$ where $\ov{u}=|\om|^{-1}\int_\om u$. The only difference sits in getting from estimate $(4.28)$ in  \cite{AMR2003} to
an estimate in terms of the frequency function $F[g]$. In this case, by standard elliptic estimates,
\[\|\nabla u_0\|_{L^2(\om)}=\|\nabla (u-\ov{g})\|_{L^2(\om)}\leq C \|g-\ov{g}\|_{H^{1/2}(\der\om)},\]
where $C$ depends on $Q_0$, $Q_1$ and $C_0$, and
\[\|u_0\|_{L^2(\der\om)}=\|g-\ov{u}\|_{L^2(\der\om)}\geq\|g-\ov{g}\|_{L^2(\der\om)},\]
hence, from estimate $(4.28)$ in \cite{AMR2003}, we get
\[\frac{\|\nabla u_0\|_{L^2(\om)}}{\|u_0\|_{L^2(\om)}}\leq C\frac{\|\nabla u_0\|^2_{L^2(\om)}}{\| u_0\|^2_{L^2(\der\om)}}\leq C\left(\frac{\|g-\ov{g}\|_{H^{1/2}(\der\om)}}{\|g-\ov{g}\|_{L^2(\der\om)}}\right)^2=CF[g]^2.\]
The rest of the proof is as in \cite{AMR2003}.
\end{proof}

\begin{lem}\label{quc}
With the same assumptions of Theorem \ref{teo1}, given $\om^\prime\subset\subset\om$ with $\mbox{dist}\left(\der\om,\om^\prime\right)\geq d_0$, there exist positive constants $K_1$ and
$K_2>1$ and $d_1\leq d_0$, depending only on $C_0$, $C_1$, $C_2$, $\mu_2$, $C_3$ and $d_0$, such that, for every $x_0\in\om^\prime$ and $r\leq d_1$,
\begin{equation}
	\label{1-8}
	\int_{B_r(x_0)}|\nabla u|^2\geq \frac{r^{K_1}}{K_2}.
\end{equation}
\end{lem}
\begin{proof}
As in the proof of Theorem 4.3 in \cite{AMR2003}, starting from the doubling inequality by Garofalo and Lin (\cite{GL}) and using Caccioppoli and Poincar\`{e} inequalities, we can show that there is a constant $2d_1\leq d_0$, depending only on $C_0$ and $C_1$, such that
we get, for $r\leq\rho\leq d_1$
\begin{equation}
	\label{6-9}
	\int_{B_\rho(x_0)}|\nabla u|^2\leq C \left(\frac{2\rho}{r}\right)^{K-2} \int_{B_r(x_0)}|\nabla u|^2,
\end{equation}
where $C$ and $K$ depend on $C_0$ and $C_1$, and $K$ depends also, increasingly, on
\begin{equation}
	\label{3-9}
	\tilde{N}(d_1)=\frac{d_1^2\int_{B_{2d_1}(x_0)}|\nabla u|^2}{\int_{B_{2d_1}(x_0)}(u-\ov{u}_r)^2},
\end{equation}
where $\ov{u}_r=\frac{1}{|B_r(x_0)|}\int_{B_r(x_0)}u$.
Now, in order to estimate $\tilde{N}(d_1)$ from above in terms of a-priori information, we can use Caccioppoli inequality
to get
\begin{equation}
	\label{1-10}
	\tilde{N}(d_1)\leq \frac{C_C}{4}\frac{\int_{B_{2d_1}(x_0)}|\nabla u|^2}{\int_{B_{d_1}(x_0)}|\nabla u|^2}.
\end{equation}

By \eqref{1-10} and Lemma \ref{lemma2.2} we have
\begin{equation}
	\label{3-11}
	\tilde{N}(d_1)
	\leq\frac{C_C}{4}\frac{\int_{\om}|\nabla u|^2}{C_{d_1}\int_{\om}|\nabla u|^2}=\frac{C_C}{4C_{d_1}}.
\end{equation}

By \eqref{6-9} with $\rho=d_1$, \eqref{3-11} and Lemma \ref{lemma2.2} again, we finally get
\begin{equation}
	\label{4-11}
	\int_{B_r(x_0)}|\nabla u|^2\geq  C^{-1} \left(\frac{r}{d_1}\right)^{K-2} \int_{B_{d_1}(x_0)}|\nabla u|^2\geq
	\frac{r^{K_1}}{C}
	\int_\om|\nabla u|^2.
\end{equation}
Now, by assumption \eqref{h5.5-1}, trace estimates and Poincar\`{e} inequality we get
\begin{equation}\label{comelolevo}
\mu_2^{-2}\leq\|g-\ov{g}\|_{H^{1/2}}^2\leq C\|u-\ov{g}\|^2_{H^1(\om)}\leq C\|\nabla u\|^2_{L^2(\om)},
\end{equation}
hence, by combining \eqref{4-11} and \eqref{comelolevo}, we finally get \eqref{1-8}.\end{proof}

\textit{Proof of Theorem \ref{teo1}.}
The first step in the proof relies on  Lemma 2.1 contained in \cite{A}. Although such Lemma 2.1 is stated in a two-dimensional setting, its validity can be extended in a straightforward manner to any dimension.

By Lemma 2.1 of \cite{A}, for any $\theta\in(0,1/2)$ there is a constant
$K_0>0$, depending only on $C_0$, $C_1$, $C_2$, $Q_0$, $Q_1$ and $\theta$, such that
\begin{equation}
	\label{1-3}
\int_\om|a-b|\left|\nabla u\right|^2\leq K_0\left(\|u-v\|_{L^2(\om)}^\theta+\|a-b\|_{L^\infty(\der\om)}\right).
\end{equation}

We now reproduce an argument due to Sincich \cite[Proposition 4.9]{S}. Let us set $\phi=a-b$ and let $x_0\in \om^\prime$ such that
\begin{equation}
\label{1-12}	
|\phi(x_0)|=\max_{\overline{\om^\prime}}|\phi(x)|.
\end{equation}

Since $a$ and $b$ satisfy  assumption \eqref{3.5-1},
\begin{equation}\label{2-12}
		|\phi(x_0)|\leq 	|\phi(x)| +2C_1r,\textrm{ for every }x\in B_r(x_0),\textrm{ with }0<r\leq d_0.
\end{equation}

Multiplying \eqref{2-12} by $|\nabla u(x)|^2$ and integrating with respect to $x$ on $B_r(x_0)$, we get
\begin{equation}\label{3-12}
	|\phi(x_0)|\int_{B_r(x_0)}\!\!\!\!\!\!|\nabla u(x)|^2dx\leq \int_{B_r(x_0)}	\!\!\!\!\!\!|\phi(x)||\nabla u(x)|^2dx  +C_1r\int_{B_r(x_0)}\!\!\!\!\!\!|\nabla u(x)|^2dx,
\end{equation}
hence
\begin{equation}\label{4-12}
	|\phi(x_0)|\leq \frac{\int_{B_r(x_0)}	|\phi(x)||\nabla u(x)|^2dx}{\int_{B_r(x_0)}|\nabla u(x)|^2dx}+2C_1r.
\end{equation}
By \eqref{1-3} and \eqref{1-8},  and  by \eqref{1-12}
we have
\begin{equation}\label{5-12}
\max_{\overline{\om^\prime}}|a(x)-b(x)|\leq K_0 K_2 r^{-K_1}\left(\|u-v\|_{L^2(\om)}^\theta+\|a-b\|_{L^\infty(\der\om)}\right)+2C_1r.
\end{equation}
By choosing an appropriate $r\in (0,d_0)$ we get \eqref{tesiteo1}.\qed

Let us now show that, by choosing a boundary condition with some additional features, we can bound form below the norm of $\nabla u$ in a neighborhood of the boundary. The following Lemma is a variation on themes treated in \cite[Lemma 2.8, Theorem 4.1]{AN}.

\begin{lem}\label{Hopf}
Let $\om$ be diffeomorphic to the unit ball and let $g$ be quantitatively unimodal, according to Definition \ref{qu}.
 If $u$ is the unique solution of problem \eqref{Pu} for a coefficient $a$ satisfying assumption \eqref{3-1}, then
\begin{equation}\label{tshopf}
|\nabla u|\geq C,\textrm{ for every }x\in \om,\textrm{ with }dist(x,\der\om)\leq\rho,
\end{equation} where $\rho>0$ and $C>0$ depend only on $C_0$, $Q_0$, $Q_1$, $m$, $M$ and $\omega$.
\end{lem}
\begin{proof}
The diffeomorphism $F$ in assumption \eqref{diff} transforms the elliptic equation in \eqref{Pu} into a similar elliptic equation in $B_1(0)$ with $W^{1,\infty}$ main coefficient. The constant of ellipticity and all the constant appearing in the assumptions \eqref{maxmin}, \eqref{ptimaxmin} and \eqref{quantunimod} shall be changed in a controlled manner only depending on the a-priori information. For this reason we  assume, without loss of generality, that $\om=B_1(0)$.

By regularity estimates for solutions of elliptic equations, the
$C^{1,\beta}(\overline{B_1(0))}$ norm of $u$ is bounded in terms of the a-priori information, hence, for $x\in B_1(0)$ and $dist(x,\Gamma_M)<\eta$ we have
\[u(x)-m\geq M-m-C\eta.\]
By choosing $\eta$ small enough, we get
\[u(x)-m\geq \frac{M-m}{2}.\]

\noindent
By Harnack inequality (\cite[Theorem 8.20, Corollary 8.21]{GT}),
\begin{equation}
	\label{hopf1} u(x)-m\geq C_\eta\frac{M-m}{2},\textrm{ for every }x\in\overline{B_{1-\eta}(0)}.
\end{equation}
In particular, if we choose $y\in\Gamma_m$ and $x=(1-\eta)y$, we get
\[u(x)-u(y)\geq C_\eta\frac{M-m}{2}.\]
By Hopf lemma (
\cite[Lemma 3.4]{GT}) 
 we have
\[|\nabla u(y)|\geq k>0,\textrm{ for every }y\in\Gamma_m.\]

Since we can proceed in the same way on $\Gamma_M$ and by using again the regularity $C^{1,\beta}$ of $u$ up to the boundary, we have
 \begin{equation}\label{hopf2}|\nabla u(x) |\geq k-C\delta^\beta
,\textrm{ for every }x\in\der\om,\textrm{ with }dist(x,\Gamma_m\cup\Gamma_M)\leq \delta.\end{equation}
By choosing $\delta=\overline{\delta}$ so that $C\overline{\delta}^\beta=k/2$, by \eqref{hopf2} and \eqref{quantunimod}, we have
\[|\nabla u|\geq \max\{\omega(\overline{\delta}),k/2\},\textrm{ on }\der\om.\]
By using again the $C^{1,\beta}$ regularity of $u$ up to the boundary, we get \eqref{tshopf}.
\end{proof}

\begin{theo}\label{teo2}
Let $\om$, $a$, $b$, $g$ and $k$ be as in Theorem \ref{teo1}. Let us also assume that
$g$  is quantitatively unimodal.
Let $u$ and $v$ in $W^{1,2}(\om)$ be the unique solutions to boundary value problems \eqref{Pu} and \eqref{Pv}.
Given $\theta\in(0,1/2)$ there are positive constants
$\tilde{C}$ and $\beta$, depending only on $C_0$, $C_1$, $C_2$, $\mu_2$, $C_3$, $Q_0$, $Q_1$, $d_0$, $\theta$, $m$, $M$ and $\omega$  such that
\begin{equation}
	\label{tesiteo2}
\|a-b\|_{L^\infty(\om)}\leq \tilde{C}\left(\|u-v\|_{L^2(\om)}^\theta+\|a-b\|_{L^\infty(\der\om)}\right)^\beta.
\end{equation}
\end{theo}
\begin{proof}
The proof follows the same steps of the proof of Theorem \ref{teo1}. It suffices to extend estimate \eqref{1-8} to every point $x_0$ in $\om$. This extension is possible by Lemma \ref{Hopf} and by the regularity assumptions on $\der\om$.
 \end{proof}
\section{Proof of Theorem \ref{theoh-2}}
\label{Sec4}
We proceed as in \cite{RGZ} (and in \cite{Ba-Uh}) and show that, for a a fixed set of coefficients, the ratio of two solutions (corresponding to different boundary values) satisfies a partial differential equation.

\begin{prop}\label{proph-1}
For $j=1,2$, the function
\begin{equation}
	\label{h9-1}
	U^{(j)}=\frac{H_2^{(j)}}{H_1^{(j)}}
\end{equation}
satisfies the equation
\begin{equation}
	\label{h10-1}
	  -\dive\left(a^{(j)}\nabla U^{(j)}\right)=0\mbox{ in }\om,
\end{equation}
where
\begin{equation}\label{h11-1}
	a^{(j)}=\frac{D^{(j)}}{\left(\sigma^{(j)}\right)^2}\left(H_1^{(j)}\right)^2=
D^{(j)}\left(u_1^{(j)}\right)^2.
\end{equation}
Moreover,
\[
	U^{(j)}=\frac{g_2}{g_1}\mbox{ on }\der\om.
\]
\end{prop}
\textit{Proof of Proposition \ref{proph-1}.}
For the sake of simplicity we drop the superscript $(j)$.
It is an easy calculation to check that, since $u_i$ solves equation \eqref{h6-1}, then
	\[-\dive\left(Du_1^2\nabla \left(\frac{u_2}{u_1}\right)\right)=0,
\]
hence, by \eqref{h8-1} and \eqref{h11-1}
	\[-\dive\left(a\nabla \left(\frac{H_2}{H_1}\right)\right)=-\dive\left(Du_1^2\nabla \left(\frac{u_2}{u_1}\right)\right)=0.
\]
\qed

Our aim is to apply Theorem \ref{teo2} to the functions
$U^{(1)}$ and $U^{(2)}$ introduced in the previous proposition.
First of all, let us show that:
\begin{claim}\label{45}If \eqref{h1-2} holds, then
\begin{equation}
	\label{h1-3}
	\left\|U^{(1)}-U^{(2)}\right\|_{L^2(\om)}\leq C\ep,
\end{equation}
where $C$ depends only on $\lambda_0$, $\lambda_1$, $\mu_0$, $E_0$ and $E_1$.
\end{claim}

\textit{Proof of Claim \ref{45}.}
Let us write
\begin{equation}
	\label{h2-3}
	 U^{(1)}-U^{(2)}=\frac{\left(H_2^{(1)}-H_2^{(2)}\right)H_1^{(2)}+\left(H_1^{(2)}-H_1^{(1)}\right)H_2^{(2)}}{H_1^{(1)}H_1^{(2)}}.
\end{equation}
We need to get a lower estimate for the functions $H_1^{(1)}$ and $H_1^{(2)}$. By \eqref{h8-1} and \eqref{h4-1},
it is enough to show that $u_1^{(j)}$ is bounded from below in terms of a-priori information.

By \eqref{h4-1} and \eqref{h5-1}, we can apply Maximum Principle and get
\begin{equation}
	\label{h1-4}
	0\leq u_1^{(j)}(x)\leq \mu_1,\textrm{ for every }x\in\om.
\end{equation}

By Theorem 8.33 in \cite{GT}, $\nabla u_1$ is bounded in terms of a-priori information.
Since $u_1^{(j)}=g_1\geq \mu_1^{-1}$ on $\der\om$ (by \eqref{h5-1}), there exists a positive constant $d$, such that
\begin{equation}
	\label{h3-4}
	u_1^{(j)}\geq \frac{\mu_1^{-1}}{2},\textrm{ in }\om\setminus\om_d,
\end{equation}
where, we recall, $\om_d=\{x\in\om\,:\, dist(x,\der\om)>d\}$.
By Harnach inequality (\cite[Theorem 8.20, Corollary 8.21]{GT}),
\begin{equation}
	\label{h5-4}
	C_H\inf_{\om_{d/2}}u_1^{(j)}\geq \sup_{\om_{d/2}} u_1^{(j)}\geq \frac{\mu_1^{-1}}{2},
\end{equation}
hence, by \eqref{h3-1} and \eqref{h5-1},
\begin{equation}
	\label{h6-4}
	\inf_\om u_1^{(j)}=\min\left\{\inf_{\om_{d/2}}u_1^{(j)}, \inf_{\om\setminus\om_d}u_1^{(j)}\right\}\geq
\min\left\{\frac{\mu_1^{-1}}{2C_H},\frac{\mu_1^{-1}}{2}\right\}:=\mu_3^{-1}.
\end{equation}
Inequality \eqref{h6-4} holds for $u_2^{(j)}$ as well.
By \eqref{h6-4} and \eqref{h4-1},
\begin{equation}
	\label{h7-4}
	H_i^{(j)}=\sigma^{(j)}u_i^{(j)}\geq (\lambda_1\mu_3)^{-1}.
\end{equation}
Moreover, by \eqref{h4-1} and \eqref{h1-4}
\begin{equation}
	\label{h8-4}
	H_i^{(j)}\leq \lambda_1 \mu_1.
\end{equation}
By \eqref{h2-3}, \eqref{h7-4}, \eqref{h7-4} and \eqref{h1-2} we finally have
\begin{equation}
	\label{h1-5}
	\|U^{(1)}-U^{(2)}\|_{L^2(\om)}\leq 2 \mu_1\lambda_1^3\mu_3^2\ep.
\end{equation}
\qed

\textit{Proof of Theorem \ref{theoh-2} (Conclusion).}

We notice that, by \eqref{h11-1}, \eqref{h2-1}, \eqref{h5-1}, \eqref{h3-1} and by Theorem 8.33 in \cite{GT}, coefficients $a^{(1)}$ and $a^{(2)}$ satisfy assumptions \eqref{3-1} and \eqref{3.5-1}.

Moreover,  by \eqref{h11-1}, \eqref{h5-1} and \eqref{Dbordo},
\begin{equation}
	\label{h3-5}
	\left\|a^{(1)}-a^{(2)}\right\|_{L^\infty(\der\om)}\leq \mu_1^2 \ep^\prime.
\end{equation}
Hence, by Theorem \ref{teo2},
\begin{equation}
	\label{h4-5}
	\|a^{(1)}-a^{(2)}\|_{L^\infty(\om)}\leq \tilde{C}\left(\ep+\ep^\prime\right)^{\theta},
\end{equation}
where $\tilde{C}$ and $\theta$ depend only on the a-priori information.

%

We now proceed as in \cite{RGZ}. By straightforward calculation it is easy to show that the function $u_1^{(j)}$ solves
\begin{equation}\label{h1-7}
	\left\{
  \begin{array}{rl}
    -\dive\left(a^{(j)}\nabla \left(\frac{1}{u_1^{(j)}}\right)\right)=H_1^{(j)}& \mbox{ in }\om,\\[2mm] 
    \displaystyle{\frac{1}{u_1^{(j)}}=\frac{1}{g_1}}& \mbox{ on }\der\om,
  \end{array}
\right.
\end{equation}
from which we can get
\begin{equation}
	\label{h1.5-7}
	 -\dive\left(a^{(1)}\nabla \left(\frac{1}{u_1^{(1)}}-\frac{1}{u_1^{(2)}}\right)\right)=H_1^{(1)}-H_1^{(2)}+
	\dive\left(\left(a^{(1)}-a^{(2)}\right)\nabla \left(\frac{1}{u_1^{(2)}}\right)\right).
\end{equation}
By Theorem 8.34 in \cite{GT},
\begin{equation}
	\label{h2-7}
	\left\|\nabla\left(\frac{1}{u_1^{(2)}}\right)\right\|_{L^\infty(\om)}=
	\left\|\frac{\nabla u_1^{(2)}}{\left(u_1^{(2)}\right)^2}\right\|_{L^\infty(\om)}
	\leq \mu_3^2\left\|\nabla u_1^{(2)}\right\|_{L^\infty(\om)}\leq C,
\end{equation}
hence, by \eqref{h4-5}, by assumption \eqref{h1-2} and by Corollary 8.7 in \cite{GT}, since $\frac{1}{u_1^{(1)}}-\frac{1}{u_1^{(2)}}=0$ on $\der\om$, we conclude that
\begin{equation}
	\label{h1-8}
	\left\|\frac{1}{u_1^{(1)}}-\frac{1}{u_1^{(2)}}\right\|_{W^{1,2}(\om)}
\leq C\left(\ep+ \ep^\prime\right)^{\theta}.
\end{equation}

Finally, by definition \eqref{h8-1},
\begin{equation}
	\label{h1-9}
	\sigma^{(1)}-\sigma^{(2)}=\frac{H_1^{(1)}}{u_1^{(1)}}-\frac{H_1^{(2)}}{u_1^{(2)}}
	=\frac{H_1^{(1)}-H_1^{(2)}}{u_1^{(1)}}+H_1^{(1)}\left(\frac{1}{u_1^{(1)}}-\frac{1}{u_1^{(2)}}\right),
\end{equation}
hence, by \eqref{h1-2} and \eqref{h1-8},
\begin{equation}
	\label{h2-9}
	\|\sigma^{(1)}-\sigma^{(2)}\|_{L^2(\om)}\leq
\mu_3\ep+\lambda_1\mu_1C\left(\ep+ \ep^\prime\right)^{\theta}.
\end{equation}
Since
\begin{equation}
	\label{h3-9}
	D^{(1)}-D^{(2)}=\frac{a^{(1)}}{\left(u_1^{(1)}\right)^2}-\frac{a^{(2)}}{\left(u_1^{(2)}\right)^2},
\end{equation}
by \eqref{h4-5}, \eqref{h6-4}  and \eqref{h1-8}
\begin{equation}
	\label{h4-9}
		\|D^{(1)}-D^{(2)}\|_{L^2(\om)}\leq C\left(\ep+ \ep^\prime\right)^{\theta},
\end{equation}
where $C$ and $\theta$ depend only on the a-priori information.

\section*{Acknowledgements}
GA was supported by FRA2014 \textit{Problemi inversi per PDE, unicit\`a, stabilit\`a, algoritmi}, Universit\`a degli Studi di Trieste.
MDC, EF and SV were partially supported by the Gruppo Nazionale per l'Analisi
Matematica, la Probabilit\`a e le loro Applicazioni (GNAMPA) of the Istituto
Nazionale di Alta Matematica (INdAM). EF was partially supported by the Research Project FIR 2013 \textit{Geometrical and qualitative aspects of PDE's}.
\bibliographystyle{plain}

\end{document}